\documentclass[12pt,a4paper]{amsart}
\usepackage{amsmath,amssymb,a4,color}
\usepackage{pgf,tikz, cases, IEEEtrantools}
\usetikzlibrary{arrows}
\usetikzlibrary{svg.path}
\usepackage{hyperref, enumerate, enumitem}
\usepackage[left=2cm,right=2cm,top=2cm,bottom=2cm]{geometry}
\usepackage{pgffor}
\usepackage{comment}

\theoremstyle{plain}
\newtheorem{thm}{Theorem}[section]

\theoremstyle{definition}

\newtheorem{remark}[thm]{Remark}

\newcommand{\eR}{\mathbb{R}}
\newcommand{\eN}{\mathbb{N}}
\newcommand{\Ve}{\mathbb{V}}
\newcommand{\Uu}{\mathbb{U}}
\newcommand{\ve}{\mathbf{v}}
\newcommand{\we}{\mathbf{w}}
\newcommand{\uu}{\mathbf{u}}
\newcommand{\z}{\tilde{z}}
\newcommand{\ra}{\tilde{r}}
\def\eqn#1$$#2$${\begin{equation}\label#1#2\end{equation}}

\numberwithin{equation}{section}

\title[Hausdorff measure of critical set for Luzin $N$ condition]{Hausdorff measure of critical set for Luzin $N$ condition}
\author{Anna Dole\v{z}alov\'a}
\address{Faculty of Mathematics and Physics,
Charles University, Sokolovsk\' a 83, Praha, Czech Republic}
 \email{dolezalova@karlin.mff.cuni.cz}
 \author{Marika Hrube\v{s}ov\'a}
 \address{Faculty of Economics, University of South Bohemia, Studentsk\' a 13, \v Cesk\' e Bud\v ejovice, Czech Republic}
 \email{mhrubesova@ef.jcu.cz}
 \author{Tom\' a\v s Roskovec}
 \address{Faculty of Economics, University of South Bohemia, Studentsk\' a 13, \v Cesk\' e Bud\v ejovice, Czech Republic;  Department of Mathematics, University of Hradec Kr\' alov\' e, Rokitansk\'eho 62, 500 03 Hradec Kr\'alov\'e, Czech Republic}
 \email{troskovec@ef.jcu.cz}
\thanks{The third author was supported by the grant GA \v CR 20-19018Y}
\thanks{The first author was partially supported by the grant GA\v CR P201/18-07996S of the Czech
Science Foundation and by the Charles University, project GA UK No. 480120}


\begin{document}

\begin{abstract}
It is well-known that there is a Sobolev homeomorphism $f\in W^{1,p}([-1,1]^n,[-1,1]^n)$ for any $p<n$ which maps a set $C$ of zero Lebesgue $n$-dimensional measure onto the set of positive measure. We study the size of this critical set $C$ and characterize its lower and upper bounds from the perspective of Hausdorff measures defined by a general gauge function.
\end{abstract}
\maketitle


\section{Introduction}

\subsection{Motivation and history}

By $\Omega$ we denote a domain, $n$ stands for dimension and $\mathcal{L}_n$ denotes Lebesgue $n$-dimensional measure. A function $f\colon\Omega\to \eR^n$, $\Omega\subseteq \eR^n$, is said to satisfy the \textit{Luzin $N$ condition} if, for every $E\subseteq\Omega$, we have
$$ \mathcal{L}_n(E)=0 \implies \mathcal{L}_n\left(f(E)\right)=0.$$
Analogously,
$f$
fulfils the \textit{Luzin $N^{-1}$ condition} if, for every $E\subseteq\Omega$, we have
$$\mathcal{L}_n\left(f(E)\right)=0 \implies \mathcal{L}_n(E)=0.$$ These are crucial properties in mechanics of solids and other physical models. The Luzin $N$ condition (also known as the Lusin property or the $N$ property) prohibits the ``creation of matter'' by deformation and the Luzin $N^{-1}$ condition prohibits the ``disappearance of matter''. From the mathematical point of view, these conditions are bound to the question of validity of change of variables formula with minimal regularity requirements, see \cite[Theorem~8.4]{BojIwa1983}, \cite[Theorem~2.5, Chapter~5]{GRbook} and \cite{Haj1993}. Also, for Sobolev spaces the validity of the Luzin $N$ condition is equivalent to validity of area formula, see \cite{R1989} and \cite{M1994}, while the Luzin $N^{-1}$ condition is equivalent to co--area formula, see \cite{MSZ}, \cite{M2001} or \cite[Section~A.~8]{MFDbook}.

Concerning the characterization of the validity of the Luzin $N$ condition, Reshetnyak \cite{Res1} proved the validity of the condition $N$ for Sobolev homeomorphisms in $W^{1,n}$,  Marcus and Mizel \cite{MaMi} proved its validity for Sobolev mappings in $W^{1,p}$ for $p>n$. To show the optimality of these results, Ponomarev \cite{Ponomarev} (see also later \cite{Ponomarev2}) provided a Sobolev homeomorphism violating the Luzin $N$ condition for $W^{1,p}$, $1\leq p<n$ and Mal\'y and Martio \cite{MaMa} used the older Cesari construction \cite{Ces} to get a continuous $W^{1,n}$ mapping violating the Luzin $N$ condition.

The characterization of the validity of the Luzin $N^{-1}$ condition differs a lot. It is possible to construct a homeomorphism that compresses a set in order to map a set of positive measure onto the set of zero measure in any $W^{1,p}$. The Sobolev norm is not crucial, so the concept of distortion and the class of the mappings with finite distortion is needed. The positive result and its optimality are given by Kauhanen, Koskela, and Mal\'y in \cite{KaKoMa-Luzin} and \cite{KoMa-N-1}, some border cases are further covered by Kleprl\'ik in \cite{Klep}.

Let us mention that the counterexample constructions violating the Luzin $N$ condition by Ponomarev and by Cesari are fundamentally different. The counterexample violating the Luzin $N^{-1}$ condition is based on the construction by Ponomarev. 

\subsection{Ponomarev construction and its refinements}
We focus on the example given by Ponomarev, i.e., the homeomorphism in $W^{1,p}([-1,1]^n,[-1,1]^n)$ for $p<n$, which maps Cantor type set $C_A$ of measure zero onto Cantor type set $C_B$ of positive measure. The detailed construction is presented in Section \ref{SekcePonomarev}. The original construction considers the Lebesgue $n$-dimensional measure and Sobolev space $W^{1,p}$. However, this may be refined, we may ask about the size of the small Cantor type set $C_A$ with respect to the Hausdorff measure $\mathcal{H}^h$ based on a gauge function $h$. We also may consider the homeomorphism in some other spaces, in general, in some spaces strictly bigger than $W^{1,n}$ and defined in finer than Sobolev scales, such as grand Sobolev spaces $ W^{1,n)}$ (see \ref{SekcePreliminaries}. Preliminaries for definition) or Sobolev--Orlicz spaces. The choice of grand Sobolev space $W^{1,n)}$ is optimal in some sense in the perspective of spaces based on integrability of weak derivative, see \cite{KaKoMa-Luzin}.

It is well-known that the Hausdorff dimension of $C_A$ may be zero (see \cite{MaMa}), but Kauhannen \cite{Kauh} also studied the largest possible size of $C_A$. Obviously, the Hausdorff measure for gauge function $h(t)=t^n$ should be still zero, otherwise the example does not violate the Luzin $N$ condition. It was shown in \cite{MaMa} that for $f\in W^{1,n}$ we can always find a critical set $C$ of Hausdorff dimension 0 such that outside $C$ the Luzin $N$ condition holds. On the other hand, for $f$ in the grand Sobolev space $W^{1,n)}$, Kauhanen \cite{Kauh} showed that the Hausdorff dimension of the critical set $C_A$ can always exceed any number below $n$ by his choice of gauge function $h_s(t)=t^n \log^s\log(4+1/t)$ for $s>0$. He constructed the Ponomarev-type homeomorphism such that $0<\mathcal{H}^{h_s}(C_A)<\infty.$ No optimality of choice $h_s$ is discussed, but it answers the question of the possible Hausdorff dimension of the exceptional set, as there is no universal constant $d<n$ such that for each $f\in W^{1,n)}$ the Luzin $N$ condition holds if we omit a set of Hausdorff dimension $d$. We study the Hausdorff measure of the critical set in more general scales, not only the powers, resulting in the following statement:

\begin{thm}\label{th:velkyponomarev}
Let $Q_0=[-1,1]^n$, $\tau:(0,\infty)\to [1,\infty)$ be a monotone, continuous function such that $\lim_{t\to 0+}\tau(t)=\infty$ and for all  $p\in (0,1]$ there exists $x_p \in(0,1)$ such that for all $t\in (0,x_p)$ we have
$$
\frac{1}{\tau(pt)}>t^n.
$$
Let $h(t):[0,\infty)\to[0,\infty)$ be a gauge function, i.e., continuous non-decreasing function such that $h(0)=0$, and satisfying $h(t)=t^n\tau(t)$ on $(0,\infty)$. Then there exists a homeomorphism $f:Q_0\to Q_0$ such that
\begin{enumerate}
    \item $f$ is the identity on the boundary of $Q_0$,
    \item $f\in W^{1,n)}(Q_0, Q_0)$,
    \item $Jf>0$ a.e. in $Q_0$,
    \item if $E\subseteq Q_0$ with $\mathcal{H}^h(E)=0$, then $\mathcal{L}_n(f(E))=0$,
    \item there exists a set $C_A$ such that $\mathcal{H}^h(C_A)\in(0,\infty)$, $\mathcal{L}_n(C_A)=0$ and $\mathcal{L}_n(f(C_A))>0$.
\end{enumerate}

\end{thm}
This is especially interesting for $\tau(t)$ being a slowly decreasing function for small $t$, such as $\log\log\log\dots (1/t)$. We can get as close to the power-type gauge function $h(t)=t^n$ as desired.

We also study the other endpoint of the Hausdorff scale. Past results are claiming the size of the exceptional set to be possibly very small, but up to our knowledge, the results consider only gauge functions in the form of power $h(t)=t^\alpha$. We prove that the exceptional set $C_A$ can be small in any possible scale of gauge functions.

\begin{thm}\label{th:malyponomarev}
Let $Q_0=[-1,1]^n$ and let $h:[0,\infty)\to [0,\infty)$ be a gauge function, i.e., continuous non-decreasing function such that $h(0)=0$. Then there exists a homeomorphism $f:Q_0\to Q_0$ such that
\begin{enumerate}
    \item $f$ is identity on the boundary of $Q_0$,
    \item $f\in W^{1,n)}(Q_0, Q_0)$,
    \item $Jf>0$ a.e. in $Q_0$,  
    \item there exists a set $C_A\subseteq Q_0$ such that $\mathcal{H}^h(C_A)=0$, $\mathcal{L}_n(C_A)=0$ and $\mathcal{L}_n(f(C_A))>0$.
\end{enumerate}
\end{thm}

This theorem is interesting for $h$ very rapidly increasing near $0$, typically with infinite one-sided derivative. We can construct a Ponomarev-type homeomorphism such that the critical set violating the Luzin $N$ condition is of measure 0 for the corresponding Hausdorff measure. This result naturally extends the well-known statement of possible Hausdorff dimension of the critical set being 0.

\subsection{Further applications of the Luzin {\it N} condition and related questions}
Let us introduce some closely related topics, applications, and development. We intend to promote papers and books that are essential for the topic, but we also point out some less known recent results.

From the historical point of view, the Peano curve \cite{Peano} presented in 1890 is probably the oldest and the most known case of violating the Luzin $N$ condition in some sense and the Cesari construction \cite{Ces} can be interpreted as Peano curve.

A question close to the Luzin $N^{-1}$ condition is the validity of Morse--Sard theorem in various settings, based on works of Morse \cite{Morse} and Sard \cite{Sard1}. In a simplified version, it states that for a sufficiently smooth function, the image of the set where Jacobian is zero should be of zero Lebesgue measure. This principle has been extended, relaxed, and developed in many directions and applications. Naturally, one wishes to state the size of the image more subtly, as in the Hausdorff dimension. One can transfer the case from the Euclidean space into the manifolds, see \cite{Sard2}. Also, the assumption of $C^k$ smoothness may be relaxed, so Lipschitz mappings \cite{ABC}, H\"older spaces \cite{BHS} Sobolev spaces \cite{dePasc, Fig2008} or BV spaces \cite{BKK} are also studied. Note that this list is picking just some highlights, and many other particular settings and applications were published recently, such as the application to PDEs in chemistry \cite{Zator} or the application in studies of the Besicovitch--Federer projection theorem \cite{Gal, GH}.

The other closely related question is the problem of the composition of operators and the regularity of the inverse operator. The composition may produce unlikely outcomes if the Luzin $N$ or $N^{-1}$ condition is not met. The boundedness and integrability of the distortion are studied to provide the validity of the Luzin $N^{-1}$ condition. We recommend classical books on this topic  \cite{AIM}, \cite{MFDbook},  \cite{Ric1993}, \cite{Vai1971}, and \cite{R1989}.

Another topic involving the Luzin $N$ condition is the question of the equivalence between the pointwise Jacobian and the distributive Jacobian, first asked by Ball \cite{Bal1977}. It is interesting as it may help to relax a lot of techniques and proofs. This question was addressed by M\"uller \cite{M1990}, by Iwaniecz and Sbordone \cite{IwaSbo1992}, and by Greco \cite{Gre1993} mostly by integrability properties. The integrability requirements may be significantly relaxed in case of the validity of the $N$ condition, as shown by D'Onofrio, Hencl, Mal\'y, and Schiattarella \cite{DHMS} based on the previous research by Henao and Mora-Corral \cite{HM-C}.

There is also a very interesting way to fail both of these conditions with such a restrictive setting as a Sobolev or even bi-Sobolev homeomorphism satisfying $Jf=0$ a.e. Such examples can map full measure set to zero measure set and zero measure set to full measure set. Also, these mappings provide a tool to construct other homeomorphisms with highly counter-intuitive properties concerning the preservation of matter or orientation, the change of the sign of the Jacobian and others, see \cite{Hkosoctverce}, \cite{DHSkosoctverce}, \cite{Ckosoctverce}, \cite{FM-CO}, \cite{Oliva2016} or \cite{LiMa}.

At the end of this section, we shortly present recent development concerning the research of the Luzin $N$ condition itself. For the survey of the development, see Koskela, Mal\'y, and Z\"urcher \cite{KMZ1}. For refinement by studying the modulus of continuity and the size of the critical set, see \cite{KMZ2}. The paper concerning the failure of the Luzin $N$ condition by Kauranen and Koskela \cite{KaKo} extended the classical result \cite{MaMa} and it was also later used by Zapadinskaya \cite{Zap2014} to transfer the knowledge from Euclidian case into more general metric measure spaces. Also, the counterexample of Ponomarev is refined with additional regularity such that it still violates the Luzin $N$ condition (see \cite{Ros1}) or the $N^{-1}$ condition (see \cite{KlMoRo}). In papers studying the Luzin $N$ condition in view of Hausdorff dimension, term  $(\alpha-\beta)$ $N$ {\it condition} is used, see \cite{Alberti2012, FeroneKoRo}.

\section{Preliminaries}\label{SekcePreliminaries}
By a {\it gauge function} $h(t):[0,\infty)\to[0,\infty)$ we denote a function satisfying
\begin{enumerate}
    \item $h$ is non-decreasing,
    \item $h(0)=0$,
    \item $h$ is continuous.
\end{enumerate}

By the {\it Hausdorff measure} $\mathcal{H}^{h}(A)$ of set $A\subseteq\mathbb{R}^n$ we understand
$$\mathcal{H}^{h}(A)=\lim_{\delta\to0_+}\left(\inf\left\{\sum_{i=1}^\infty h(\operatorname{diam} U_i): A\subseteq \bigcup_i U_i; \operatorname{diam}(U_i)<\delta \right\}\right).$$

The definition may slightly differ in literature. The limit $\delta\to 0_+$ can be replaced by supremum, and we may consider the covering system $U_i$ by both general open sets and balls.
Note that for the most classical case $h(t)=t^\alpha$ we write $\mathcal{H}^\alpha$ instead of $\mathcal{H}^{t^\alpha}$. 

By the {\it Hausdorff dimension} of set $A$ we understand a non-negative parameter $d$, such that
$$\dim_H(A)=\inf_{d\geq0}\{\mathcal{H}^{d}(A)=0\}.$$

We claim that our examples belong to the {\it grand Sobolev space} $W^{1,n)}$. This space is introduced in \cite{IwaSbo1992} by Iwaniec and Sbordone and we refer to \cite{DonSboSch2013} for a survey of the notion. The {\it grand Lebesgue norm} is
$$\|f\|_{q)}=\sup_{0<\varepsilon<q-1}\left(\frac{\varepsilon}{|\Omega|}\int_{\Omega}|f|^{q-\varepsilon} \right)^{\frac{1}{q-\varepsilon}}.$$
This norm defines the {\it grand Lebesgue space} $L^{q)}(\Omega)$, a Banach function space that is very close to $L^{q}$, the sharp inclusions explaining the relations between function spaces of interest are
$$
L^q(\Omega)\subsetneq L^q\log^{-1}(L)(\Omega)\subsetneq L^{q)}(\Omega)\subsetneq\bigcap_{\alpha>1} L^q\log^{-\alpha}(L)(\Omega)\subsetneq \bigcap_{1<p<q}L^p(\Omega).
$$
The grand Sobolev space is a set of such functions that the function itself and all its partial derivatives up to the desired rank belong to the corresponding grand Lebesgue space. We emphasize that usage of this modern tool allows for sharpening our result, see this also in \cite{Kauh}.

In this text we use notation $A\lesssim B$ and $A\approx B$. By $A\lesssim B$ we denote that there exists a constant $K$ independent on parameters and depending only on the dimension and gauge function $h$ such that $A \leq K B$. $A\approx B$ denotes both $A\lesssim B$ and $B\lesssim A$ hold. We use notation $\|x\|_\infty$ for the maximum norm of the vector $x$ and $Q_{a,r}=\{x\in\mathbb{R}^n: \|x-a\|_{\infty}<r\}$ for open $n$-dimensional cube of center $a$ and edge length $2r$.

\section{Ponomarev construction}\label{SekcePonomarev}
We describe the Ponomarev construction in general way with notation consistent with its description in \cite{MFDbook}. We obtain Sobolev homeomorphism $f:(-1,1)^n\to(-1,1)^n$ with $Jf>0$ a.e. violating the Luzin $N$ condition.

Let $\Ve$ be vertices of cube $[-1,1]^n$. Let $\Ve^k=\Ve\times\Ve\times\dots\times\Ve$, $k\in\eN$, be a set of indices and let us consider two strictly decreasing sequences $a_k$ and $b_k$ such that
\begin{enumerate}
    \item $a_0=1$, $b_0=1$,
    \item $\lim_{k\to \infty}a_k=0$,
    \item $\lim_{k\to \infty}b_k>0$.
\end{enumerate}
Note that this setting aims to break the Luzin $N$ condition. In order to break the Luzin $N^{-1}$ condition we demand $\lim_{k\to\infty}a_k>0$ and $\lim_{k\to\infty}b_k=0$ instead. However, in order to make the resulting mapping interesting, we  have to set $a_k$ and $b_k$ carefully and check the crucial property, the integrability of distortion.

Let us define $z_0=\z_0=0$ and
$$
r_k=2^{-k}a_k\text{ and }\ra_k=2^{-k}b_k.
$$
We start with $Q(z_0,r_0)=(-1,1)^n$ and proceed by induction. For ${\ve}=[v_1,v_2,v_3,\dots, v_k]\in{\Ve}^k$ we denote ${\we}({\ve})=[v_1,v_2\dots v_{k-1}]\in {\Ve}^{k-1}$ and we define
$$
z_{{\ve}}=z_{{\we}}({\ve})+\frac{r_{k-1}}{2}v_k=z_0+\sum_{i=1}^{k}\frac{r_{i-1}}{2}v_i.
$$
For simplicity we write ${\we}$ instead of ${\we}({\ve})$. Around center $z_{\ve}$ we define an outer and inner cube
$$
Q'_{\ve}=Q\left(z_{\ve}, \frac{r_{k-1}}{2}\right)\text{ and }Q_{\ve}=Q\left(z_{\ve}, r_k\right).
$$

\begin{figure}[h t p]

\begin{tikzpicture}[scale=1]
\newcommand\Square[1]{+(-#1,-#1) rectangle +(#1,#1)}
 
    \foreach \x in {-1, 1}
    \foreach \y in {-1, 1}
    \draw (\x,\y) \Square{1}; 
 
   \foreach \x in {-1, 1}
    \foreach \y in {-1, 1}
    \draw[fill= lightgray] (\x,\y) \Square{0.6}; 

\draw[->] (2.2,0)--(3.2,0);
  \node[above] at (2.7, 0) {$f_1$};
  
  \foreach \x in {-1,1}
    \foreach \y in {-1, 1}
    \draw (\x+5.5,\y) \Square{1}; 
 
   \foreach \x in {-1,1}
    \foreach \y in {-1, 1}
    \draw[fill= lightgray] (\x+5.5,\y) \Square{0.8};

    \foreach \x in {-1, 1}
    \foreach \y in {-1, 1}
    \draw (\x,\y-5) \Square{1}; 
 
   \foreach \x in {-1.3,-0.7,0.7, 1.3}
    \foreach \y in {-1.3,-0.7,0.7, 1.3}
    \draw (\x,\y-5) \Square{0.3}; 
    
   \foreach \x in {-1.3,-0.7,0.7, 1.3}
    \foreach \y in {-1.3,-0.7,0.7, 1.3}
    \draw[fill= lightgray] (\x,\y-5) \Square{0.08};

\draw[->] (2.2,-5)--(3.2,-5);
  \node[above] at (2.7, -5) {$f_2$};
  
  \foreach \x in {-1,1}
    \foreach \y in {-1, 1}
    \draw (\x+5.5,\y-5) \Square{1}; 
 
   \foreach \x in {-1.4,-0.6,0.6, 1.4}
    \foreach \y in {-1.4,-0.6,0.6, 1.4}
    \draw (\x+5.5,\y-5) \Square{0.4};
    
   \foreach \x in {-1.4,-0.6,0.6, 1.4}
    \foreach \y in {-1.4,-0.6,0.6, 1.4}
    \draw[fill=lightgray] (\x+5.5,\y-5) \Square{0.3};

\end{tikzpicture}
\caption{First two steps in Ponomarev construction of $f$. Above: $f_1$ maps $\bigcup_{{\ve}\in{\Ve}} Q_{\ve}$ onto $\bigcup_{{\ve}\in{\Ve}} \tilde{Q}_{\ve}$. Below: $f_2$ maps $\bigcup_{{\ve}\in{\Ve^2}} Q_{\ve}$ onto $\bigcup_{{\ve}\in{\Ve^2}} \tilde{Q}_{\ve}$.}
\label{obrazek_konstrukce}
\end{figure}
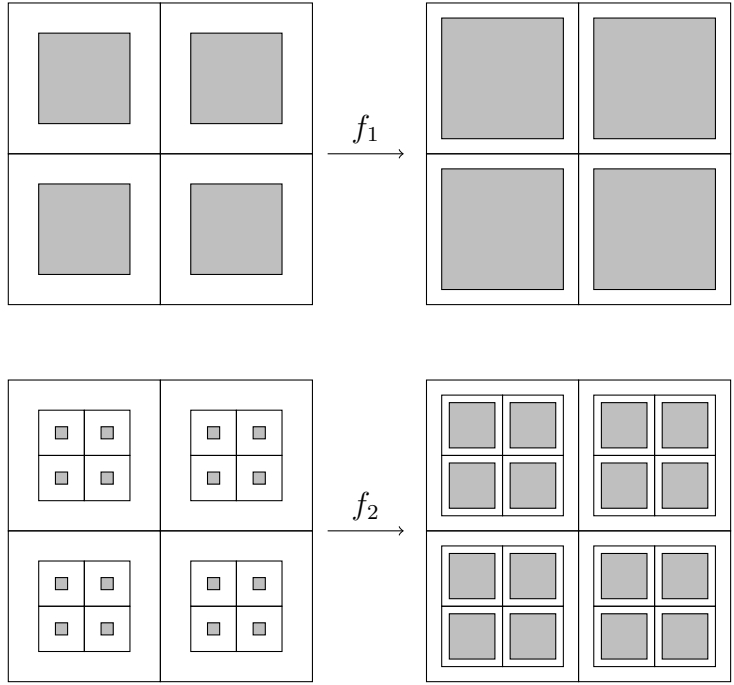

In the $k$-th step of the construction, we use indices ${\ve}\in {\Ve}^k$ and produce $2^{nk}$ cubes $Q_{\ve}$, which are copies of the same cube.

We get a Cantor-type set $C_A$ defined as
$$
C_A=\bigcap_{k=1}^\infty\bigcup_{{\ve}\in{\Ve}^k}Q_{\ve}=C_a\times C_a\times\dots\times C_a,
$$
where $C_a$ are one-dimensional Cantor-type sets. Its construction is illustrated on the left-hand side of Figure \ref{obrazek_konstrukce}.

Analogously for the image we define the first cube as $\tilde{Q}(z_0,r_0)=(-1,1)^n$ and centers as
$$
\z_{{\ve}}=\z_{{\we}}+\frac{\ra_{k-1}}{2}v_k=\z_0+\sum_{i=1}^{k}\frac{\ra_{i-1}}{2}v_i,
$$
and we define a structure of cubes by
$$
\tilde{Q}'_{\ve}=Q(\z_{\ve}, \frac{\ra_{k-1}}{2})\text{ and }\tilde{Q}_{\ve}=Q(\z_{\ve}, \ra_k).
$$
We further define $C_B$ as
$$
C_B=\bigcap_{k=1}^\infty\bigcup_{{\ve}\in{\Ve}^k}\tilde{Q}_{\ve}=C_b\times C_b\times\dots\times C_b,
$$
where $C_b$ are again one-dimensional Cantor-type sets.

Concerning the Lebesgue measure of both $C_A$ and $C_B$, we obtain
$$
\mathcal{L}_n(C_A)=\lim_{k\to \infty}\mathcal{L}_n\left(\bigcup_{{\ve}\in{\Ve}^k} Q_{\ve}\right) =\lim_{k\to \infty}2^{nk}(2r_k)^n=\lim_{k\to \infty}2^{nk-nk}2^na_k^{n}=0,
$$
$$
\mathcal{L}_n(C_B)=\lim_{k\to \infty}2^{nk}(2\ra_k)^n=\lim_{k\to \infty}2^{nk-nk}2^nb_k^n=2^n(\lim_{k\to \infty} b_k)^n>0.
$$

Our goal is to define a sequence of homeomorphism $f_k:[-1,1]^n\to[-1,1]^n$ such that its limit $f$ is a homeomorphism mapping $C_A$ onto $C_B$, as we demonstrate in Figure \ref{obrazek_konstrukce}. We start with $f_0(x)=x$. To define $f_1$, we map $Q_\ve$ onto $\tilde{Q}_\ve$ homogenously with respect to the centres $z_{\ve}$ and $\z_{\ve}$ for all ${\ve}\in{\Ve}$. We define $f_1$ from $Q'_{\ve}\setminus Q_{\ve}$ onto $\tilde{Q}'_{\ve}\setminus \tilde{Q}_{\ve}$ radially for the supremum norm with respect to the centres $z_{\ve}$ and $\z_{\ve}$.
In the general step, we keep $f_k=f_{k-1}$ on $[-1,1]^n\setminus(\bigcup_{{\ve}\in{\Ve}^{k}}Q'_{\ve})$. It remains to define $f_k$ inside copies of $Q'_{{\ve}}$. We use the homogeneous mapping of $Q_{\ve}$ onto $\tilde{Q}_{\ve}$ and the radial mapping of $Q'_{\ve}\setminus Q_{\ve}$ onto $\tilde{Q}'_{\ve}\setminus \tilde{Q}_{\ve}$, both with respect to centres $z_{\ve}$ and $\z_{\ve}$, see Figure \ref{obrazek_jeden_ctverec}.

\begin{figure}[h t p]

\begin{tikzpicture}[scale=1]
\newcommand\Square[1]{+(-#1,-#1) rectangle +(#1,#1)}
 
    \draw (1,0) \Square{1}; 
    \draw[fill= lightgray] (1,0) \Square{0.4}; 
    \draw[ultra thin,<->](0.6,0)--(1.4,0);
    \draw[ultra thin,<->](1,-0.4)--(1,0.4); 
    \draw[dotted](0.6,-0.4)--(0,-1);
    \draw[dotted](0.6,0.4)--(0,1);
    \draw[dotted](1.4,-0.4)--(2,-1);
    \draw[dotted](1.4,0.4)--(2,1);
    
    \draw[->] (2.2,0)--(3.2,0);
    \node[above] at (2.7, 0) {$f_k$};
    \draw (5,0) \Square{1.5}; 
    \draw[fill= lightgray] (5,0) \Square{1.2}; 
    \draw[ultra thin,<->](3.8,0)--(6.2,0);
    \draw[ultra thin,<->](5,-1.2)--(5,1.2); 
    \draw[dotted](3.8,-1.2)--(3.5,-1.5);
    \draw[dotted](3.8,1.2)--(3.5,1.5);
    \draw[dotted](6.2,-1.2)--(6.5,-1.5);
    \draw[dotted](6.2,1.2)--(6.5,1.5);

\end{tikzpicture}
\caption{Mapping $f_k$ transforms $Q_{\ve}$ onto $\tilde{Q}_{\ve}$ (the gray area) and $Q'_{\ve}\setminus Q_{\ve}$ onto $\tilde{Q}'_{\ve}\setminus \tilde{Q}_{\ve}$ (the white area), ${\ve}\in{\Ve}^k$.}
\label{obrazek_jeden_ctverec}
\end{figure}
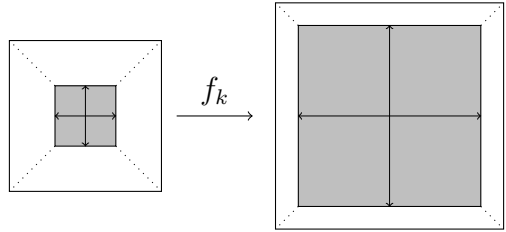

Formally, we define
\begin{equation*}
f_k(x)=\begin{cases}f_{k-1}(x)&\text{ for }x\notin\bigcup_{{\ve}\in{\Ve}^{k}}Q'_{\ve},\\
f_{k-1}(z_{\ve})+(\alpha_k\|x-z_{\ve}\|_\infty+\beta_k)\frac{x-z_{\ve}}{\|x-z_{\ve}\|_\infty}&\text{ for }x\in Q'_{\ve}\setminus Q_{\ve}, {\ve}\in{\Ve}^k, \\
f_{k-1}(z_{\ve})+\frac{\ra_k}{r_k}(x-z_{\ve})&\text{ for }x\in Q_{\ve}, {\ve}\in{\Ve}^k,
\end{cases}
\end{equation*}
where $\alpha_k$ and $\beta_k$ are chosen for $f_k$ to map the annulus $Q'_{\ve}\setminus Q_{\ve}$ onto the annulus $\tilde{Q}'_{\ve}\setminus \tilde{Q}_{\ve}$, i.e, such that
\begin{equation}\label{z ceho alpha-beta}
\alpha_k r_k+\beta_k=\ra_k\text{ and }\alpha_k \frac{r_{k-1}}{2}+\beta_k=\frac{\ra_{k-1}}{2}.
\end{equation}
Note that such $f_k$ maps \begin{equation}\label{Kam zobrazim Qv}\bigcup_{{\ve}\in{\Ve}^j} Q_{\ve}\text{ onto }\bigcup_{{\ve}\in{\Ve}^j} \tilde{Q}_{\ve}
\end{equation}
for all $j\leq k$. Since $f_k$ is continuous and one-to-one mapping between compact spaces, it is a homeomorphism.

We need to estimate the derivatives of $f_k$ in $Q_{\ve}$ and $Q'_{\ve}\setminus Q_{\ve}$ for ${\ve}\in{\Ve}^k$. For $x\in Q_{\ve}$ we get
$$
|Df_k|=\frac{\ra_k}{r_k}=\frac{b_k}{a_k}.
$$
For $x\in Q'_{\ve}\setminus Q_{\ve}$ we should consider two possible directions of partial derivatives, based on which coordinate determinates the norm $\|x-z_{{\ve}}\|_\infty$. Without loss of generality, suppose it is the first coordinate. For $x\in Q'_{\ve}\setminus Q_{\ve}$ we estimate
\begin{equation}\label{odhad prac. derivace}
\begin{aligned}|D_{x_1}f_k|&=\left|\left( (\alpha_k\|x-z_{\ve}\|_\infty+\beta_k)\frac{x-z_{\ve}}{\|x-z_{\ve}\|_\infty}\right)_{x_i}\right|\leq \alpha_k,\\
|D_{x_i}f_k|&=\left|\left( (\alpha_k\|x-z_{\ve}\|_\infty+\beta_k)\frac{x-z_{\ve}}{\|x-z_{\ve}\|_\infty}\right)_{x_i}\right|\leq \alpha_k+\frac{\beta_k}{\|x-z_v\|_\infty},\quad i\neq 1.
\end{aligned}\end{equation}
Therefore, each mapping $f_k$ belongs to $W^{1,\infty}$ (however, the sequence is not bounded there).

The limit mapping $f$ is absolutely continuous on almost all lines which are parallel to the coordinate axes, since almost all lines do not intersect the Cantor set $C_A$ and hence $f$ is Lipschitz on such lines. Also, $f$ maps $C_A$ onto $C_B$, based on \eqref{Kam zobrazim Qv}. Its pointwise partial derivatives on $Q'_{\ve}\setminus Q_{\ve}$ for ${\ve}\in{\Ve}^k$ are the same as those of $f_k$. In the end, we estimate 
$$
\|Df\|^p_p=\sum_{k=1}^\infty\sum_{{\ve}\in{\Ve}^k}\int_{Q'_{\ve}\setminus Q_{\ve}}|Df|^p.
$$ 
We should also check that Jacobian is positive almost everywhere. Since $J_f$ is equal to $J_{f_k}$ on set $Q'_{\ve}\setminus Q_{\ve}$ and the union of these sets has full measure, it is enough to verify the positivity of $J_{f_k}$, which can be done by a straightforward calculation. 
\begin{remark}
The choice $a_k=\tfrac{1}{k+1}$ and $b_k=\tfrac{1}{2} (1+\tfrac{1}{k+1})$ provides pointwise estimate $Df\lesssim k$ for $x\in Q'_{\ve}\setminus Q_{\ve}$, $\mathcal{L}_n(|Q'_{\ve}\setminus Q_{\ve})\approx2^{-nk}\frac{1}{k^{n+1}}$ and $|Df|\in L^p$ is finite if $p<n$. Note that these estimates can be adjusted to the special choice of $a_k$ and $b_k$ and they differ in literature.
\end{remark}

\section{Proof of Theorem  \ref{th:velkyponomarev} and Theorem \ref{th:malyponomarev}}

We now present the estimate for the norm of the derivative for a fairly general choice of $a_k$ and $b_k$. We show that for this choice, the resulting mapping belongs to grand Sobolev space.

Let $a_k$ be an arbitrary monotone positive sequence with $a_0=1$ and $\lim_{k\to\infty}a_k=0$ and set 
$$
b_k=\frac{1}{2}(1+a_k).
$$
This together with \eqref{z ceho alpha-beta} implies 
$$
\alpha_k=2^{-1}\text{ and }\beta_k=2^{-k-1}
$$
for $k\geq 1$. 
For further use we prepare the pointwise estimate for partial derivative of $f_k(x)$ for $x\in Q'_{\ve}\setminus Q_{\ve}$, $\ve\in{\Ve}^k$ based on \eqref{odhad prac. derivace}, we get
$$
|Df_k(x)|= \displaystyle\max_{i\in\{1, \dots, n\}}\{|D_{x_i}f_k|\}=\operatorname{max}\left\{\alpha_k,\alpha_k+\frac{\beta_k}{\|x-z_v\|_\infty}\right\}\lesssim\frac{\beta_k}{\|x-z_v\|_\infty}.
$$
The following estimate is universal for both Theorem \ref{th:velkyponomarev} and Theorem \ref{th:malyponomarev} and may be used for any $a_k, b_k$ satisfying properties above. Using the fact that for each $k$ we have $2^{nk}$ annuli with the same size, between which the function differs only by translation, we calculate

\begin{align*}
\sup_{0<\varepsilon\leq n-1} & \varepsilon\int_{(-1,1)^n}|Df|^{n-\varepsilon} =   \sup_{0<\varepsilon\leq n-1}\varepsilon\left(\sum_{k=1}^\infty \sum_{\ve\in\Ve^k}\int_{Q'_\ve\setminus Q_\ve}|Df_k|^{n-\varepsilon}\right) \\
&\lesssim \sup_{0<\varepsilon\leq n-1}\varepsilon\sum_{k=1}^\infty 2^{nk}\int_{Q\left(0,\frac{r_{k-1}}{2}\right)\setminus Q\left(0,r_k\right)} \left(\frac{\beta_k}{\|x\|_\infty}\right)^{n-\varepsilon}\, dx\\
&\lesssim \sup_{0<\varepsilon\leq n-1}\varepsilon\sum_{k=1}^\infty 2^{nk}\int_{2^{-k}a_k}^{2^{-k}a_{k-1}} \left(\frac{2^{-k-1}}{t}\right)^{n-\varepsilon}t^{n-1}\, dt\\
&\lesssim \sup_{0<\varepsilon\leq n-1}\varepsilon\sum_{k=1}^\infty 2^{nk}\int_{2^{-k}a_k}^{2^{-k}a_{k-1}} 2^{(-k-1)(n-\varepsilon)}t^{-1+\varepsilon}\, dt\\
&\lesssim \sup_{0<\varepsilon\leq n-1}\varepsilon\sum_{k=1}^\infty \left(2^{(k+1)\varepsilon}\left[\varepsilon^{-1}t^\varepsilon\right]_{2^{-k}a_k}^{2^{-k}a_{k-1}}
\right)\lesssim \sup_{0<\varepsilon\leq n-1}\sum_{k=1}^\infty 2^{(k+1)\varepsilon}\left(2^{-\varepsilon k}a_{k-1}^\varepsilon - 2^{-\varepsilon k}a_k^\varepsilon 
\right)\\
&\lesssim \sup_{0<\varepsilon\leq n-1}2^{\varepsilon}\sum_{k=1}^\infty (a_{k-1}^\varepsilon-a_k^\varepsilon)\lesssim \sup_{0<\varepsilon\leq n-1} \left(a_0^\varepsilon - \lim_{k\to \infty}a_k^\varepsilon\right) = \sup_{0<\varepsilon\leq n-1}a_0^\varepsilon<\infty,
\end{align*}
since the limit of $a_k^\varepsilon$ is zero. Therefore both terms are finite and $f\in W^{1,n)}((-1,1)^n)$.

\begin{proof}[Proof of Theorem \ref{th:malyponomarev}]
We choose $a_k$ satisfying the conditions above (i.e., monotone positive with limit 0) such that
$$
h(c_n 2^{-k} a_k)<2^{-2nk},
$$
where $c_n=2\sqrt{n}$. We can do so, since $h$ is non-decreasing continuous and $\lim_{t\to 0_+}h(t)=0$. Set $b_k=1/2(1+a_k)$ as before. The Ponomarev type construction described in Section \ref{SekcePonomarev} ensures the properties (1) and (3) and the choice of parameters gives us (2).
It remains to prove (4). Since 
$$
C_A\subseteq \bigcup_{{\ve}\in{\Ve}^k}Q_{\ve}
$$
for an arbitrary $k$, from the definition of Hausdorff measure we have
$$
\mathcal{H}^{h}(C_A)\leq\lim_{k\to\infty}\sum_{{\ve}\in{\Ve}^k} h(\operatorname{diam} Q_{\ve})=\lim_{k\to\infty}2^{nk} h(c_n r_k)<\lim_{k\to\infty}2^{-nk}=0.
$$
Also $\mathcal{L}_n(C_A)=0$ and $\mathcal{L}_n(f(C_A))=\mathcal{L}_n(C_B)>0$ as was shown in Section \ref{SekcePonomarev}.
\end{proof}

\begin{proof}[Proof of Theorem \ref{th:velkyponomarev}]
The proof is divided into several steps.
\begin{enumerate}[label=(\roman*)]

\item Choice of $a_k$

\noindent   
We claim that we can find a decreasing sequence $a_k$ satisfying the properties from Section \ref{SekcePonomarev} such that $a_k^n\tau(2^{-k}c_n a_k)\approx 1$. 
Since $\tau$ is continuous and bounded by $1$ from below, for every parameter $p$ there has to be point $t_p\in (0,1]$ such that $1/\tau(pt_p)=t_p^n$ and $1/\tau(pt)>t^n$ on $(0, t_p)$. We set $a_k = t_{2^{-k}}$. To show that it is a monotone sequence, let us have $p_1>p_2$ and elaborate. From monotonicity of $\tau$ we have
$$
t_{p_1}^n=\frac{1}{\tau(p_1 t_{p_1})}> \frac{1}{\tau(p_2 t_{p_1})}.
$$
This implies that $t_{p_2}$ must be smaller than $t_{p_1}$, since $1/\tau(p_2t)>t^n$ on $(0, t_{p_2})$. Now choose $\varepsilon>0$ and find $p$ small enough such that $1/\tau(pt)<\varepsilon^n$ for $t\in(0,1]$. Since $t_p\in(0,1]$, we have $t_p^n<\varepsilon^n$. This ensures that the limit of $a_k$ is zero.
With this choice of the sequence $a_k$, for any ${\Uu}$ subsystem of ${\Ve}^k$ we obtain
\begin{equation}\label{podsystem_krychli}
\sum_{{\uu}\in{\Uu}} h(\operatorname{diam} Q_{\uu})= \#{\Uu} h(2^{-k}c_n a_k)\approx \#{\Uu} 2^{-nk}a_k^n\tau(2^{-k}c_n a_k) \approx 2^{-nk} \#{\Uu},
\end{equation}
where $\#{\Uu}$ denotes the number of elements of the system.

\item Properties (1) -- (3)

\noindent
By setting $b_k=(1+a_k)/2$ and proceeding as before, we obtain a Sobolev homeomorphism $f$ which satisfies properties (1) -- (3).

\item $0<\mathcal{H}^h(C_A)<\infty$

\noindent
We immediately see from \eqref{podsystem_krychli} that the Hausdorff measure of $C_A$ is finite, since
$$
\mathcal{H}^{h}(C_A)\leq \lim_{k\to\infty}\sum_{{\ve}\in{\Ve}^k} h(\operatorname{diam} Q_{\ve})\approx \lim_{k\to\infty} 2^{-nk} \#{\Ve}^k = 1.
$$
The other inequality is proven in several steps. We mimic the proof from \cite[Lemma 3.2]{Kauh}. 
Since $C_A$ is a compact set, it is enough to prove that for any finite open covering $\{U_j\}$ of $C_A$ we have 
\begin{equation}\label{pokryti_nerovnost}
\sum_{j} h(\operatorname{diam} U_j)\gtrsim  1.
\end{equation}
We may assume that there exists $x_j\in C_A\cap U_j$ for each $j$, then $U_j\subseteq B_j=B(x_j,\operatorname{diam} U_j)$. Therefore
\begin{align*}
\sum_{j} h(\operatorname{diam} U_j)=\sum_{j} h(\operatorname{diam} B_j/2)= \sum_{j} 2^{-n}(\operatorname{diam} B_j)^n\tau(\operatorname{diam} B_j/2)\\
\geq \sum_{j} 2^{-n}(\operatorname{diam} B_j)^n\tau(\operatorname{diam} B_j)=\sum_{j} 2^{-n}h(\operatorname{diam} B_j),
\end{align*}
so we may consider only coverings by balls in \eqref{pokryti_nerovnost}.
We now wish to show that for every $l\in\mathbb{N}$ and $j$ we have
$$
\sum_{{\ve}\in{\Ve}^l, Q_{\ve}\subseteq B_j}h(\operatorname{diam}Q_{\ve})\lesssim h(\operatorname{diam} B_j).
$$
This can be proven by taking $Q_{\ve}\subseteq B_j$ for some ${\ve}\in{\Ve}^l$ and $m$ the smallest integer such that $Q_{\uu}\subseteq B_j$ for some ${\uu}\in{\Ve}^m$ (obviously, $m\leq l$). Set
$$
{\Uu}=\{{\uu}\in{\Ve}^m: Q_{\uu}\cap B_j\neq\emptyset\}.
$$
Since $B_j$ is centered at a point from $C_A$, from the definition of $m$ we obtain
$$
r_m\lesssim \operatorname{diam} B_j \lesssim r_{m-1}.
$$
Therefore there exists an upper bound for the number of disjoint cubes of side length $r_{m-1}$, which have a non-empty intersection with $B_j$, and this upper bound is independent of $j$ and $m$. Since the size of ${\Uu}$ is at most $2^{n}$ times this number, we have an (independent) upper bound for $\#{\Uu}$ too.

Together with \eqref{podsystem_krychli} it provides
\begin{align*}
h(\operatorname{diam} B_j) &\geq h(\operatorname{diam}Q_{\uu}) \gtrsim \sum_{{\uu}\in{\Uu}} h(\operatorname{diam} Q_{\uu}) \approx 2^{-nm}  \#{\Uu} \\
&= 2^{-nl} \# \{{\ve}\in{\Ve}^l: Q_{\ve}\subseteq Q_{\uu}, {\uu}\in{\Uu}\} \\
&\approx \sum_{{\uu}\in{\Uu}} \sum_{\substack{{\ve}\in{\Ve}^l, \\ Q_{\ve}\subseteq Q_{\uu}}} h(\operatorname{diam} Q_{\ve}) \geq \sum_{{\ve}\in{\Ve}^l, Q_{\ve}\subseteq B_j} h(\operatorname{diam}Q_{\ve}).
\end{align*}
Finally, since $C_A$ is compact, there exists $k_0$ such that for every $Q_{\ve}\in{\Ve}^k$, $k\geq k_0$, we can find $j$ such that $Q_{\ve}\subseteq B_j$. For such $k$ we have
$$
\sum_j h(\operatorname{diam} B_j) \gtrsim \sum_j \sum_{\substack{{\ve}\in{\Ve}^k, \\ Q_{\ve}\subseteq B_j}}h(\operatorname{diam} Q_{\ve}) \geq \sum_{{\ve}\in{\Ve}^k} h(\operatorname{diam} Q_{\ve})\approx 1.
$$
This combined gives us the desired property that $\mathcal{H}^h(C_A)>0$. Combined with the fact that $f(C_A) = C_B$ we have (5) (the Lebesgue measure properties are obvious from previous sections).

\item Mapping $z$

\noindent
For each point $x\in C_A$ we can find ${\ve}_x$ from ${\Ve}^\mathbb{N}$ such that 
$$x=\bigcap_{j} Q_{({\ve}_x)_j}
$$
and this is a one-to-one correspondence. Let $\pi_i$ denote the projection of ${\ve}\in{\Ve}$ to its $i$-th coordinate. Define $c_i: C_A\to \{0,1\}^\mathbb{N}$ which (in each coordinate) tells whether we chose a cube ``on the right-hand side or on the left-hand side'', i.e., 
$$
c_i(x) = \{\pi_i(({\ve}_x)_j)\}_{j=1}^\infty.
$$
Next consider a function $bin: \{0,1\}^\mathbb{N}\to [0,1]$, which takes ${\ve}$ and interprets it as the number $0.{\ve}_1 {\ve}_2 \dots$ written in the binary system. This is obviously onto, however, it is not injective (because both $(0,1,1,1,\dots)$ and $(1,0,0,0,\dots)$ are mapped to $1/2$).
We denote 
$$
z(x) = (bin(c_1(x)),\dots,bin(c_n(x))):C_A\to [0,1]^n.
$$
Then $z$ is onto and it is injective outside of the set 
\begin{align*}
S&=\{ x\in C_A: bin(c_i(x))=k/2^j \text{ for some } i\in\{1,\dots,n\}, j\in\mathbb{N}_0  \text{ and } k\in\{0,\dots,2^j\} \}\\
&=\{ x\in C_A: c_i(x) \text{ is constant from some index } j_0 \in\mathbb{N} \text{ for some } i\in\{1,\dots,n\} \},
\end{align*}
which consists of preimages of boundaries of dyadic cubes in $[0,1]^n$.

\item Image of $\mathcal{H}^h$ under $z$

\noindent
We start with showing that $\mathcal{H}^h(S)=0$ and $\mathcal{L}_n(z(S))=0$. The second statement follows simply from the fact that the boundary of a dyadic cube is a set of (Lebesgue) measure zero and $z(S)$ is their countable union. The first statement is proven in a similar way since $S$ is a countable union of the sets in the form
$$
S_{i,j,k}= \{ x\in C_A: bin(c_i(x))=k/2^j \}
$$
for $i\in\{1,\dots,n\}, j\in\mathbb{N}_0  \text{ and } k\in\{0,\dots,2^j\}$.
These are (up to a permutation of coordinates and a translation) equal to $\{0\}\times C_a \times \dots \times C_a$ and $\mathcal{H}^h(S_{i,j,k})=0$, because $\mathcal{H}^h(C_A)<\infty$ and $C_A$ contains uncountably many disjoint copies of this set.

Now we show the equality of measures 
$$
z(\mathcal{H}^h)=\left(\mathcal{H}^h(C_A)\right)^{-1}\mathcal{L}_n.
$$
For any open dyadic cube $D$ of the edge length $2^{-j}$ take the corresponding ${\ve}\in{\Ve}^j$. Then $\mathcal{L}_n(D)=2^{-jn}$ and $\mathcal{H}^h(z^{-1}(D)) = \mathcal{H}^h(Q_{\ve}\cap C_A) = 2^{-jn}\mathcal{H}^h(C_A)$, because $z^{-1}(D) = Q_{\ve}\cap C_A \setminus S'$, where $S'\subseteq S$.

The system
$$
\mathcal{D} = \{D: D \text{ is an open dyadic cube}\}\cup\{S': S' \text{ is a measurable subset of } S \}
$$
is closed on finite intersections and the sigma algebra generated by $\mathcal{D}$ contains all Borel sets in $[0,1]^n$. Since $z(\mathcal{H}^h) = \left(\mathcal{H}^h(C_A)\right)^{-1}\mathcal{L}_n$ on elements from $\mathcal{D}$, they are the same on $[0,1]^n$.

\item Property (4)

\noindent
We can analogously construct $\tilde{z}: C_B\to [0,1]^n$ for which
$$
\tilde{z}(\mathcal{L}_n) = \left(\mathcal{L}_n(C_B)\right)^{-1}\mathcal{L}_n.
$$
Then from the fact that $f(Q_{\ve})=\tilde{Q}_{\ve}$ for an arbitrary ${\ve}$ it follows that $\tilde{z} \circ f = z$, i.e., the following diagram commutes:

\begin{figure}[h]

    \begin{tikzpicture}[scale=1]
  
   \node at (1.5,2) {$C_B$};
   \node at (0,0) {$C_A$};
   \node at (3,0) {$[0,1]^n$};
   \draw[->] (0.2,0.3)--(1.2,1.7);
   \draw[->] (0.4,0)--(2.3,0);
   \draw[->] (1.6,1.7)--(2.8,0.3);
   \node at (0.55,1.1) {$f$};
   \node at (2.55,1.1) {$\tilde{z}$};
   \node at (1.4,-0.2) {$z$};
   \draw (1.4, 0.8) circle (0.1);

    \end{tikzpicture}

\end{figure}

Because the injectivity is broken only on $S$ and $\mathcal{H}^h(S)=0$ and $\mathcal{L}_n(z(S))=0$, for an arbitrary measurable $E\subseteq C_A$ we have
$$
\mathcal{H}^h(E)=0\iff \mathcal{L}_n(z(E))=0\iff \mathcal{L}_n(\tilde{z}(f(E)))=0\iff \mathcal{L}_n(f(E))=0.
$$
The Luzin $N$ condition holds outside of $C_A$ since $f$ is locally Lipschitz there, and any set with finite measure $\mathcal{H}^h$ is of zero Lebesgue measure. Therefore property (4) holds.
\end{enumerate}
\end{proof}

\section*{Acknowledgements}
We wish to thank Stanislav Hencl for valuable comments on the topic of this paper. We would also like to thank David Hru\v{s}ka for inspiring discussions. 
\bibliographystyle{plain}

\end{document}